\documentclass[a4paper,11pt]{amsart}
\usepackage{nicola-ams}

\title{A remark on nonlocal Neumann conditions for the fractional Laplacian}
\thanks{\textit{2010 Mathematics Subject Classification}: 35S15, 35J99, 47G20}
\thanks{\textit{Keywords}: fractional Laplacian, nonlocal normal derivative,
nonlocal Neumann conditions, regional Laplacian}

\allowdisplaybreaks[1]

\begin{document}

\maketitle

% Enter the first author's name and address:
\centerline{\scshape nicola abatangelo}
\medskip
{\footnotesize
 \centerline{D\'{e}partement de Math\'{e}matique,}
 \centerline{Universit\'e Libre de Bruxelles CP 214}
 \centerline{boulevard du Triomphe}
 \centerline{1050 Bruxelles, Belgium}
}

\begin{abstract}
We show how nonlocal boundary conditions of Robin type
can be encoded in the pointwise expression of the fractional operator.
Notably, the fractional Laplacian 
of functions satisfying
homogeneous nonlocal Neumann conditions
can be expressed as a regional 
operator with a kernel
having logarithmic behaviour at the boundary.
\end{abstract}

%\tableofcontents

\section{Introduction}

The purpose of this short note is to put in evidence a special
feature of the fractional Laplacian when coupled with nonlocal Robin
boundary conditions.

By \textit{fractional Laplacian} we mean the nonlocal operator of order $2s\in(0,2)$
\begin{align}\label{ds}
\Ds u(x):=c_{n,s}\;\pv\int_{\R^n}\frac{u(x)-u(y)}{{|x-y|}^{n+2s}}\;dy
:=c_{n,s}\:\lim_{\eps\downarrow 0}\int_{\R^n\setminus B_\eps(x)}\frac{u(x)-u(y)}{{|x-y|}^{n+2s}}\;dy,
\end{align}
where $c_{n,s}$ is a positive normalizing constant.
We refer to Valdinoci and the author \cite{av} 
for an introduction to the basic traits of this operator,
with particular emphasis on the differences with the classical elliptic theory for the Laplacian,
and other possible notions of fractional Laplacian. We underline how the definition in~\eqref{ds}
makes sense\footnote{More precisely, the concerned function~$u$
must enjoy some regularity and some growth control at infinity,
in order to have the integral finite:
we skip here on these important details and we refer the 
interested reader to~\cite{av}.} only for functions defined in all of~$\R^n$: this yields an associated
boundary value problem on some domain $\Omega\subset\R^n$ which looks like 
\begin{align*}
\left\lbrace\begin{aligned}
\Ds u &= f & & \text{in }\Omega \\
u &= g & &  \text{in }\R^n\setminus\Omega.
\end{aligned}\right.
\end{align*}
In this problem, the function $g:\R^n\setminus\Omega\to\R$ accounts as a boundary condition of Dirichlet sort.
A Neumann boundary condition has been proposed by Dipierro, Ros-Oton, and Valdinoci~\cite{MR3651008},
by means of some \textit{nonlocal normal derivative} (for which we keep the original notation from the authors)
\begin{equation}\label{neu}
\widetilde\NN_s u(x):=\left(\int_\Omega\frac{dy}{{|x-y|}^{n+2s}}\right)^{-1}\int _\Omega\frac{u(x)-u(y)}{{|x-y|}^{n+2s}}\;dy,
\qquad x\in\R^n\setminus\Omega,
\end{equation}
in its normalized version, see~\cite[equation (1.5)]{MR3651008}.
%We refer to~\cite{MR3651008} for motivations to introduce this quantity.
With~\eqref{neu}, it is possible to study the elliptic boundary value problem
\begin{align*}
\left\lbrace\begin{aligned}
\Ds u &= f & & \text{in }\Omega \\
\widetilde\NN_su &= g & &  \text{in }\R^n\setminus\Omega
\end{aligned}\right.
\end{align*}
or also the parabolic one
\begin{align}\label{prob:neu-par}
\left\lbrace\begin{aligned}
\partial_tu+\Ds u &= 0 & & \text{in }\Omega\times(0,\infty) \\
\widetilde\NN_su &= 0 & &  \text{in }\big(\R^n\setminus\Omega\big)\times(0,\infty) \\
u &= u_0 & & \text{on }\Omega\times\{0\}.
\end{aligned}\right.
\end{align}

We show how the fractional Laplacian of a function satisfying homogeneous Neumann 
conditions $\widetilde\NN_su=0$ in $\R^n\setminus\Omega$ can be reformulated 
as a regional type operator, \textit{i.e.}, under the form
\begin{align}\label{regional-gen}
\Ds u(x)=c_{n,s}\:\pv\int_\Omega\big(u(x)-u(y)\big)\;K(x,y)\;dy,
\qquad x\in\Omega,
\end{align}
for some suitable measurable kernel $K(x,y):\Omega\times\Omega\to\R$.
See equations~\eqref{eq:main} and~\eqref{eq:main2} below for the detailed expression of the kernel. 

\subsection{Notations}
In the following, we will use these notations without further notice.
Fixed a nonempty set $\Sigma\subset\R^n$, $\dist(\cdot,\Sigma):\R^n\to[0,+\infty)$
stands for the distance function $\dist(x,\Sigma)=
\inf\{|x-y|:y\in\Sigma\}$. When $\Sigma=\partial\Omega$,
where $\Omega$ is the reference domain in \eqref{neu},
we simply write $d=\dist(\cdot,\partial\Omega)$.
%; mind that, to our purposes,
%$\Sigma$ can be an open set in $\R^n$, but also the boundary of an open set.
For a measurable function $f:\R^n\to\R$,
we write $\supp(f):=\overline{\{f\neq0\}}$.

The binary operations $\wedge$ and $\vee$
will respectively denote the ``$\min$'' and ``$\max$'' operations 
between real numbers:
\begin{align*}
a\wedge b=\min\{a,b\},\quad a\vee b=\max\{a,b\}\qquad a,b\in\R.
\end{align*}
%The characteristic function of a measurable $D\subset\R^n$
%is denoted by $\chi_D$, \textit{i.e.},
%$\chi_D(x)=1$ if $x\in D$ and $\chi_D(x)=0$ otherwise.
In our computations, we will also make use of the particular choice
$K(x,y)={|x-y|}^{-n-2s}$ in~\eqref{regional-gen}, yielding the usually called 
\textit{regional fractional Laplacian}
\begin{align}\label{regional}
\Ds_\Omega u(x):=c_{n,s}\:\pv\int_\Omega\frac{u(x)-u(y)}{{|x-y|}^{n+2s}}\;dy,
\qquad x\in\Omega.
\end{align}
Finally, for a measurable $\beta:\R^n\to[0,1]$, let
\begin{align}\label{k}
k_\beta(x,y):=\int_{\R^n\setminus\Omega}
\frac{1-\beta(z)}{\displaystyle{|x-z|}^{n+2s}{|y-z|}^{n+2s}\int_\Omega\frac{dw}{{|z-w|}^{n+2s}}}\;dz,
\qquad x,y\in\Omega.
\end{align}
We simply write $k(x,y)$ when $\beta=0$ in $\R^n\setminus\Omega$.% in \eqref{k}.

\subsection{Main result} The precise statements go as follows.
\begin{theorem}\label{thm}
Fix $\Omega\subset\R^n$ open, bounded, and with $C^{1,1}$ boundary.
Consider a measurable~$\beta:\R^n\to[0,1]$ such that $\beta=0$ in $\Omega$.
%Suppose $N,D\subset\R^n\setminus\Omega$ are measurable sets such that 
%$N\cap D=\emptyset$ and $N\cup D=\R^n\setminus\Omega$.
Let $u\in C^{2s+\alpha}_{loc}(\Omega)\cap L^\infty(\Omega)$, for some $\alpha>0$, satisfy
\begin{align}\label{robin}
\beta(x)u(x)+\big(1-\beta(x)\big)\widetilde\NN_s u(x)=0\quad\text{for a.e. }\ x\in\R^n\setminus\Omega.
%\qquad\text{and}\qquad u(x)=0\quad\text{for a.e. }\ x\in D.
\end{align}
Then, for $x\in\Omega$,
\begin{align*}
& \Ds u(x) = \Ds_\Omega u(x)-u(x)\Ds\beta(x)+
c_{n,s}\int_\Omega \big(u(x)-u(y)\big)\;k_\beta(x,y)\;dy \\
& \ = c_{n,s}\:\pv\int_\Omega\big(u(x)-u(y)\big)\left(\frac1{{|x-y|}^{n+2s}}+k_\beta(x,y)\right)dy
+u(x)\:c_{n,s}\int_{\R^n\setminus\Omega}\frac{\beta(y)}{{|x-y|}^{n+2s}}\;dy
\end{align*}
where $k_\beta:\Omega\times\Omega\to\R$ is given by~\eqref{k}, nonnegative, symmetric, continuous,
and for any fixed~$x\in\Omega$
there exists $C=C(x)>0$ such that
\begin{align}\label{est}
% \frac1C\big(1+|\ln\dist(y,\supp(1-\beta))|\big) \leq 
k_\beta(x,y) 
\leq C\big(1+|\ln\dist(y,\supp(1-\beta))|\big)
\qquad\text{for any }\ y\in\Omega.
\end{align}
\end{theorem}

A particular form of the mixed boundary condition in \eqref{robin}
(namely, when $\beta$ is a characteristic function)
has been used by Leonori, Medina, Peral, Primo, and Soria \cite{lmpps}
to study a principal eigenvalue problem.

In the particular case when $\beta=1$ in $\R^n\setminus\Omega$ in the above theorem,
we have 
\begin{align*}
\Ds u(x) = \Ds_\Omega u(x)+u(x)\Ds\chi_{\Omega}(x),
\end{align*}
where $\chi_\Omega$ denotes the characteristic function of $\Omega$.
Conversely, when $\beta=0$ in $\R^n$,
we entail the following.
\begin{corollary}\label{cor}
In the assumptions of Theorem \ref{thm}, if $u$ satisfies
\begin{align*}
\widetilde\NN_s u(x)=0\qquad\text{for a.e. }\ x\in\R^n\setminus\Omega,
\end{align*}
then, for $x\in\Omega$,
\begin{align}\label{eq:main}
\begin{split}
\Ds u(x)
& = \Ds_\Omega u(x)+c_{n,s}\int_\Omega\big(u(x)-u(y)\big)\,k(x,y)\;dy \\
& = c_{n,s}\:\pv\int_\Omega\big(u(x)-u(y)\big)\left(\frac1{{|x-y|}^{n+2s}}+k(x,y)\right)dy
\end{split}
\end{align}
where 
\begin{align}\label{eq:main2}
k(x,y):=\int_{\R^n\setminus\Omega}
\frac{dz}{\displaystyle{|x-z|}^{n+2s}{|y-z|}^{n+2s}\int_\Omega\frac{dw}{{|z-w|}^{n+2s}}},
\qquad x,y\in\Omega,
\end{align}
is positive, symmetric, continuous, and for any fixed $x\in\Omega$
there exists $C=C(x)>0$ such that
\begin{align}\label{est2}
%\frac1C\big(1+|\ln\dist(y,\R^n\setminus\Omega)|\big) \leq 
k(x,y) 
\leq C\big(1+|\ln\dist(y,\R^n\setminus\Omega)|\big)
\qquad\text{for any }\ y\in\Omega.
\end{align}
Moreover, it also holds
\begin{align}\label{est3}
k(x,y) \geq \frac1C\big(1+|\ln\dist(y,\R^n\setminus\Omega)|\big)
\qquad\text{for any }\ y\in\Omega.
\end{align}
\end{corollary}
This last corollary is saying that, when homogeneous nonlocal Neumann conditions are assumed,
the fractional Laplacian amounts to be a perturbation of the regional operator $\Ds_\Omega$
defined in~\eqref{regional}:
%a somewhat explicit expression of the 
%perturbing kernel $k$ is given in~\eqref{k} above. 
this nice property becomes particularly 
interesting when thinking of its stochastic repercussions.

\subsection{Stochastic heuristics}
The coupling of $\Ds$ with $\widetilde\NN_su=0$ has indeed a precise interpretation 
from the stochastic perspective. We quote from~\cite{MR3651008}:

\begin{quotation}\it
``The probabilistic interpretation of the Neumann problem (1.4)~\emph{[\eqref{prob:neu-par} in this note]} 
may be summarized as follows:
\begin{enumerate}
\item[(1)] $u(x,t)$ is the probability distribution of the position 
of a particle moving randomly inside $\Omega$.
\item[(2)] When the particle exits $\Omega$, it immediately comes back into $\Omega$.
\item[(3)] The way in which it comes back inside $\Omega$ is the following: 
If the particle has gone to $x\in\R^n\setminus\Omega$, 
it may come back to any point $y\in\Omega$, 
the probability density of jumping from $x$ to $y$ being proportional to ${|x-y|}^{-n-2s}$.
\end{enumerate}
These three properties lead to the equation (1.4)~\emph{[\eqref{prob:neu-par} in this note]}, 
being $u_0$ the initial probability distribution of the position of the particle.''
\end{quotation}
We refer to~\cite[Section 2.1]{MR3651008} for further details. 
The described process might look quite similar to a~\textit{censored process},
as introduced by Bogdan, Burdzy, and Chen~\cite{MR2006232}, 
or at least to a~\textit{stable-like process}
(following the wording of Chen and Kumagai~\cite{MR2008600}).
The class of stable-like processes is the one induced
by infinitesimal generators of the type (see \cite[equation (1.3)]{MR2008600})
\begin{align*}
\pv\int_\Omega\frac{\big(u(x)-u(y)\big)\:j(x,y)}{{|x-y|}^{n+2s}}\;dy
\end{align*}
where the kernel $j$ is supposed to be positive, symmetric, and 
bounded between two constants
\begin{align*}
\frac1C\leq j(x,y)\leq C,\qquad x,y\in\Omega.
\end{align*}
In view of Corollary~\ref{cor}, and in particular of estimate~\eqref{est2},
the process built in~\cite{MR3651008}
does not fall into the stable-like class
because of its singular boundary behaviour,
although such singularity is rather weak.

\section{Estimates}
As a standing hypothesis, we consider here $\Omega\subset\R^n$ open, bounded, and with
$C^{1,1}$ boundary satisfying the interior and the exterior sphere condition.

\begin{lemma}\label{lem:veryfirst} 
There exists~$c>0$ such that, for any $z\in\R^n\setminus\overline\Omega$,
\begin{align}\label{49687974}
\frac1c\left(d(z)^{-2s}\wedge d(z)^{-n-2s}\right)\leq \int_\Omega\frac{dw}{{|z-w|}^{n+2s}} \leq c\left(d(z)^{-2s}\wedge d(z)^{-n-2s}\right).
\end{align}
\end{lemma}
\begin{proof}
As $\Omega$ is bounded, there exists $R>0$ such that
\begin{align*}
\Omega\subset B_R(z)\setminus B_{d(z)}(z).
\end{align*}
Therefore, using polar coordinates,
\begin{align*}
\int_\Omega\frac{dw}{{|z-w|}^{n+2s}} 
& \leq \int_{B_R(z)\setminus B_{d(z)}(z)}\frac{dw}{{|z-w|}^{n+2s}} 
 =|\partial B|\int_{d(z)}^R t^{-1-2s}\;dt\leq\frac{|\partial B|}{2s}\;d(z)^{-2s}.
\end{align*}
Also
\begin{align*}
\int_\Omega\frac{dw}{{|z-w|}^{n+2s}} \leq \int_\Omega\frac{dw}{{d(z)}^{n+2s}}=|\Omega|\;d(z)^{-n-2s},
\end{align*}
from which we conclude that there exists $c>0$ such that
\begin{align*}
\int_\Omega\frac{dw}{{|z-w|}^{n+2s}} \leq c\left(d(z)^{-2s}\wedge d(z)^{-n-2s}\right).
\end{align*}
To get also the inverse inequality, we split the analysis for $d(z)$ small and large.
Indeed for $d(z)$ large one has, by Fatou's Lemma,
\begin{align*}
\liminf_{|z|\uparrow\infty}|z|^{n+2s}\int_\Omega\frac{dw}{{|z-w|}^{n+2s}}
\geq \int_\Omega\liminf_{|z|\uparrow\infty}\frac{|z|^{n+2s}}{{|z-w|}^{n+2s}}\;dw
=|\Omega|
\end{align*}
so that
\begin{align*}
\int_\Omega\frac{dw}{{|z-w|}^{n+2s}}
\geq c\;d(z)^{-n-2s}\qquad \text{for }d(z)\text{ sufficiently large.}
\end{align*}
Let us turn now to the case $d(z)$ small.
Since $\partial\Omega$ is compact and smooth, $d$ is a continuous function.
For any $z\in\R^n\setminus\Omega$ there exists $\pi(z)\in\partial\Omega$ such that $d(z)=|\pi(z)-z|$
and, by the interior sphere condition, there are $w_0\in\Omega$ and $r>0$ such that
$B_r(w_0)\subset\Omega$ and $\partial B_r(w_0)\cap\partial\Omega=\{\pi(z)\}$.
So
\begin{align*}
\int_\Omega\frac{dw}{{|z-w|}^{n+2s}} \geq \int_{B_r(w_0)}\frac{dw}{{|z-w|}^{n+2s}}.
\end{align*}
Up to a rotation and a translation, we can suppose $z=0$ and $w_0=\big(r+d(z)\big)e_1$. Then
\begin{align*}
\int_\Omega\frac{dw}{{|z-w|}^{n+2s}} \geq \int_{B_r\big((r+d(z))e_1\big)}\frac{dw}{{|w|}^{n+2s}}
=\frac1{d(z)^{2s}}\int_{B_{r/d(z)}\big(\frac{r+d(z)}{d(z)}e_1\big)}\frac{d\xi}{{|\xi|}^{n+2s}}.
\end{align*}
As $d(z)\downarrow 0$ we have the convergence
\begin{align*}
\int_{B_{r/d(z)}\big(\frac{r+d(z)}{d(z)}e_1\big)}\frac{d\xi}{{|\xi|}^{n+2s}}
\longrightarrow \int_{\{\xi_1>1\}}\frac{d\xi}{{|\xi|}^{n+2s}}
\end{align*}
so that also 
\begin{align*}
\int_\Omega\frac{dw}{{|z-w|}^{n+2s}} \geq c\;d(z)^{-2s}\qquad\text{for }d(z)\text{ sufficiently small},
\end{align*}
from which we conclude the validity of \eqref{49687974}.
\end{proof}

The following Lemma computes an upper bound on $k_\beta$.

\begin{lemma}\label{lem:first}
For a measurable $\beta:\R^n\setminus\Omega\to[0,1]$ and $x,y\in\Omega$ let
$k_\beta$ be defined as in \eqref{k}.
Then, for any~$x\in\Omega$, there exists~$C=C(x)>0$ such that
\begin{align}\label{1567464}
% \frac1C\big(1+|\ln\dist(y,\supp(1-\beta))|\big) \leq 
k_\beta(x,y) 
\leq C\big(1+|\ln\dist(y,\supp(1-\beta))|\big)
\qquad\text{for any }\ x,y\in\Omega.
\end{align}
\end{lemma}
\begin{proof}
Denote by $N:=\supp(1-\beta)$. The term
\begin{align*}
\int_\Omega\frac{dw}{{|z-w|}^{n+2s}},\qquad z\in N\subset\R^N\setminus\Omega
\end{align*}
has been taken care of with Lemma~\ref{lem:veryfirst}. In order to prove~\eqref{1567464},
we plug~\eqref{49687974} into~\eqref{k}, so that we are left with estimating
\begin{align*}
\int_{\R^n\setminus\Omega}
\frac{1-\beta(z)}{\displaystyle{|x-z|}^{n+2s}{|y-z|}^{n+2s}\left(d(z)^{-2s}\wedge d(z)^{-n-2s}\right)}\;dz
\leq\int_N
\frac{d(z)^{2s}\vee d(z)^{n+2s}}{{|x-z|}^{n+2s}{|y-z|}^{n+2s}}\;dz
\end{align*}
in which we split the integration as follows
\begin{align*}
\int_{N\cap\{d(z)<1\}}
\frac{d(z)^{2s}}{{|x-z|}^{n+2s}{|y-z|}^{n+2s}}\;dz+
\int_{N\cap\{d(z)\geq1\}}
\frac{d(z)^{n+2s}}{{|x-z|}^{n+2s}{|y-z|}^{n+2s}}\;dz.
\end{align*}
The second addend is obviously uniformly bounded (above and below away from 0) in~$x,y\in\Omega$.
Therefore, we must concentrate on the first one: since~$x$ is fixed inside~$\Omega$,
we can drop term~$|x-z|^{-n-2s}$.
Then, since~$d(z)<|y-z|$,
\begin{align*}
\int_{N\cap\{d(z)<1\}}
\frac{d(z)^{2s}}{{|y-z|}^{n+2s}}\;dz & \leq 
\int_{N\cap\{d(z)<1\}}
\frac{dz}{{|y-z|}^{n}}\leq
\int_{(\R^n\setminus B_{\dist(y,N)}(y))\cap\{d(z)<1\}}
\frac{dz}{{|y-z|}^{n}} \\
& \leq |\partial B|\left|\int_{\dist(y,N)}^{\diam(\Omega)+2}\frac{dt}{t}\right|\leq 
C\big(1+\big|\ln\dist(y,N)\big|\big)
\end{align*}
where $\diam(\Omega)=\sup\{|x-y|:x,y\in\Omega\}$.
\end{proof}

Finally, this Lemma shows how the upper bound in \eqref{1567464} is optimal.

\begin{lemma}\label{lem:second}
For a measurable $\beta:\R^n\setminus\Omega\to[0,1]$ and $x,y\in\Omega$ let
$k_\beta$ be defined as in \eqref{k}. Suppose that, for some $\eps>0$,
$N_\eps:=\{\beta<1-\eps\}$ is nonempty.
Then, for any~$x\in\Omega$, there exists~$C=C(x,\eps)>0$ such that
\begin{align}\label{15674642}
k_\beta(x,y) \geq
\frac1C\big(1+|\ln\dist(y,N_\eps)|\big) 
\end{align}
\end{lemma}
\begin{proof}
To deduce such lower bound, remark that up to a smooth change of variable,
the behaviour of the concerned integral \eqref{k} defining $k_\beta$ is the same as
\begin{align*}
\int_{\{0<z_1<1\}}
\frac{z_1^{2s}}{{|d_N(y)e_1+z|}^{n+2s}}\;dz
\end{align*}
which is easily computable as
\begin{align*}
\int_{\{0<z_1<1\}}
\frac{z_1^{2s}}{{|d_N(y)e_1+z|}^{n+2s}}\;dz &=
\int_0^1 z_1^{2s}\int_{\R^{n-1}}
\frac{dz'}{{\Big(\big(\dist(y,N_\eps)+z_1\big)^2+{|z'|}^2\Big)}^{n/2+s}}\;dz'\;dz_1 \\
&=\int_0^{1/\dist(y,N_\eps)} t^{2s}\int_{\R^{n-1}}
\frac{d\xi}{{\big({(1+t)}^2+{|\xi|}^2\big)}^{n/2+s}}\;dt \\
&= \int_0^{1/\dist(y,N_\eps)} \frac{t^{2s}}{{(1+t)}^{1+2s}}\int_{\R^{n-1}}
\frac{d\xi}{{\big(1+{|\xi|}^2\big)}^{n/2+s}}\;dt \\
&\geq c\int_1^{1/\dist(y,N_\eps)} \frac{dt}{t}=c|\ln \dist(y,N_\eps)|
\end{align*}
for some $c>0$. %This concludes the proof of \eqref{1567464}.
\end{proof}

\section{Proof of the main results}

\begin{proof}[Proof of Theorem~\ref{thm}]
In the following, consider a function $u:\R^n\to\R$ satisfying \eqref{robin},
%\begin{equation*}
%\widetilde\NN_s u(x)=0 \qquad\text{for a.e. }x\in N
%\end{equation*}
which, by \eqref{neu}, can be rewritten
\begin{equation}\label{neu2}
u(x)=\big(1-\beta(x)\big)\bigg(\int_\Omega\frac{dy}{{|x-y|}^{n+2s}}\bigg)^{-1}\int_\Omega\frac{u(y)}{{|x-y|}^{n+2s}}\;dy,
\qquad\text{for a.e. }x\in\R^n\setminus\Omega.
\end{equation}
So, for a fixed $x\in\Omega$,
\begin{align*}
\Ds u(x)
& =c_{n,s}\:\pv\int_{\R^n}\frac{u(x)-u(y)}{{|x-y|}^{n+2s}}\;dy \\
& =c_{n,s}\:\pv\int_{\Omega}\frac{u(x)-u(y)}{{|x-y|}^{n+2s}}\;dy
% +u(x)\:c_{n,s}\int_{D}\frac{dy}{{|x-y|}^{n+2s}}
+c_{n,s}\int_{\R^n\setminus\Omega}\frac{u(x)-u(y)}{{|x-y|}^{n+2s}}\;dy
\end{align*}
and we plug~\eqref{neu2} in the third addend obtaining
\begin{align}
\Ds u(x)=\ & \Ds_\Omega u(x)\ + \notag \\
& + c_{n,s}\int_{\R^n\setminus\Omega}\frac{u(x)-\displaystyle\big(1-\beta(y)\big)\bigg(\int_\Omega\frac{dz}{{|y-z|}^{n+2s}}\bigg)^{-1}\int_\Omega\frac{u(z)}{{|y-z|}^{n+2s}}\;dz}{{|x-y|}^{n+2s}}\;dy \notag \\
=\ & \Ds_\Omega u(x)+u(x)\int_{\R^n\setminus\Omega}\frac{\beta(y)}{{|x-y|}^{n+2s}}\;dy \notag \\
& +c_{n,s}\int_{\R^n\setminus\Omega}\frac{\displaystyle u(x)\int_\Omega\frac{dz}{{|y-z|}^{n+2s}}-\int_\Omega\frac{u(z)}{{|y-z|}^{n+2s}}\;dz}{\displaystyle{|x-y|}^{n+2s}\int_\Omega\frac{dz}{{|y-z|}^{n+2s}}}\;\big(1-\beta(y)\big)\;dy \notag \\
=\ & \Ds_\Omega u(x)-u(x)\Ds\beta(x)\ + \notag \\
& + c_{n,s}\int_{\R^n\setminus\Omega}\frac{\displaystyle \int_\Omega\frac{u(x)-u(z)}{{|y-z|}^{n+2s}}\;dz}{\displaystyle{|x-y|}^{n+2s}\int_\Omega\frac{dw}{{|y-w|}^{n+2s}}}\;\big(1-\beta(y)\big)\;dy \label{468} \\
=\ & \Ds_\Omega u(x)-u(x)\Ds\beta(x)\ + \notag \\
& + c_{n,s}\int_{\Omega}\big( u(x)-u(z) \big)\int_{\R^n\setminus\Omega}
\frac{1-\beta(y)}{\displaystyle{|x-y|}^{n+2s}{|y-z|}^{n+2s}\int_\Omega\frac{dw}{{|y-w|}^{n+2s}}}\;dy\;dz \label{469} \\
=\ & \Ds_\Omega u(x)-u(x)\Ds\beta(x)+c_{n,s}\int_\Omega \big(u(x)-u(z)\big)\;k_\beta(x,z)\;dz \notag
\end{align}
The exchange in the integration order from~\eqref{468} to~\eqref{469}
is justified by the Fubini's Theorem, by noticing that for~$x,z\in\Omega$ and $y\in\R^n\setminus\Omega$
(recall that $u$ and $\beta$ are bounded by assumption)
\begin{align*}
\frac{|u(x)-u(z)|\:\big(1-\beta(y)\big)}{\displaystyle{|x-y|}^{n+2s}{|y-z|}^{n+2s}\int_\Omega\frac{dw}{{|y-w|}^{n+2s}}}
\leq \frac{2\|u\|_{L^\infty(\Omega)}}
{\displaystyle{|x-y|}^{n+2s}{|y-z|}^{n+2s}\int_\Omega\frac{dw}{{|y-w|}^{n+2s}}},
\end{align*}
and
\begin{align*}
\int_\Omega\int_{\R^n\setminus\Omega}\frac{dy}
{\displaystyle{|x-y|}^{n+2s}{|y-z|}^{n+2s}\int_\Omega\frac{dw}{{|y-w|}^{n+2s}}}\;dz%\ \leq\\
%\leq\ \int_\Omega\int_{\R^n\setminus\Omega}\frac{dy}
%{\displaystyle{|x-y|}^{n+2s}{|y-z|}^{n+2s}\int_\Omega\frac{dw}{{|y-w|}^{n+2s}}}\;dz
\end{align*}
is finite by Lemma~\ref{lem:first} (with $\beta\equiv 0$),
which in turn proves also~\eqref{est}.
This concludes the proof.
\end{proof}

\begin{proof}[Proof of Corollary~\ref{cor}]
This is a straightforward consequence of Theorem~\ref{thm},
together with Lemma~\ref{lem:second} for deducing the lower bound in~\eqref{15674642}.
\end{proof}

\section*{Bibliography}
%\addcontentsline{toc}{section}{Bibliography}

\bigskip

\end{document}